\documentclass[11pt,leqno]{amsart}

\usepackage{amsmath,amsfonts,amssymb,amsthm}
\usepackage{graphicx}

\theoremstyle{plain}
\newtheorem{theorem}{Theorem}

\newtheorem{lemma}[theorem]{Lemma}

\theoremstyle{definition}
\newtheorem{definition}[theorem]{Definition}
\newtheorem{remark}[theorem]{Remark}

\newcommand{\R}{\mathbb R}
\newcommand{\C}{\mathbb C}

\newcommand{\Heis}{\mathbb H}

\newcommand{\cL}{\mathcal L}

\newcommand{\cD}{{\mathcal D}}

\newcommand{\bD}{\mathbf D}
\newcommand{\bM}{\mathbf M}
\newcommand{\bI}{\mathbf I}

\newcommand{\bi}{{\mathbf i}}

\newcommand{\fheis}{{\mathfrak h}}

\newcommand{\deriv}[1]{{\frac{\partial}{\partial #1}}}

\DeclareMathOperator{\diver}{div}
 \DeclareMathOperator{\Imag}{Im}
\DeclareMathOperator{\Ker}{Ker}

\numberwithin{theorem}{section} \numberwithin{equation}{section}

\title[$C^2$ conformal maps of the Heisenberg group are $C^\infty$]{A new proof of the $C^\infty$ regularity of $C^2$ conformal mappings on the Heisenberg group}
\author{Alex D. Austin}
\address{ADA: Department of Mathematics \\ University of California, Los Angeles \\ Box 951555 \\ Los Angeles, CA 90095-1555}
\email{aaustin@math.ucla.edu}
\author{Jeremy T. Tyson}
\address{JTT: Department of Mathematics \\ University of Illinois at Urbana-Champaign \\ 1409 West Green St. \\ Urbana, IL 61801}
\email{tyson@illinois.edu}
\date{\today}
\thanks{JTT was supported by NSF Grant DMS-1600650 `Mappings and measures in sub-Riemannian and metric spaces'.}

\begin{document}

\maketitle

\begin{center}
\it Dedicated to Bogdan Bojarski
\end{center}

\begin{abstract}
We give a new proof for the $C^\infty$ regularity of $C^2$ smooth conformal mappings of the sub-Riemannian Heisenberg group. Our proof avoids any use of nonlinear potential theory and relies only on hypoellipticity of H\"ormander operators and quasiconformal flows. This approach is inspired by prior work of Sarvas and Liu.
\end{abstract}

\section{Introduction}

In this paper we give a new proof of the $C^\infty$ regularity of $C^2$ smooth conformal mappings of the Heisenberg group.

Recall that Liouville's rigidity theorem states that conformal mappings of Euclidean domains in dimension at least three are the restrictions of M\"obius transformations. In particular, they are $C^\infty$ smooth.

Liouville's theorem has a long and storied history which is closely tied to the development of geometric mapping theory and analysis in metric spaces throughout the latter half of the twentieth century. The first proof, for $C^3$ diffeomorphisms, is due to Liouville in 1850. Gehring's proof \cite{geh:rings} of the Liouville theorem for $1$-quasiconformal mappings was a major turning point and inaugurated a line of research aimed at identifying optimal Sobolev regularity criteria. An extension of the Liouville theorem to $1$-quasiregular mappings was first obtained by Reshetnyak; see for instance his books \cite{resh:bounded-distortion} and \cite{resh:stability}. Since that time the topic has been extensively investigated by many people, including Bojarski, Iwaniec, Martin and others. The book by Iwaniec and Martin \cite{im:gft} gives an excellent overview.

Our present work is motivated by a recent proof of Liouville's theorem due to Liu \cite{liu:liouville}. In contrast with previous proofs, which relied on nonlinear PDE and the regularity theory for $p$-harmonic functions, Liu's proof uses purely linear techniques, specifically, an analysis of Ahlfors' conformal strain operator
%$Sf := \tfrac12(Df+(Df)^T)-(\tfrac1n\diver f)\cdot \bI_n$,
and quasiconformal flows. An earlier paper by Sarvas \cite{sar:liouville} used similar methods to derive Liouville's theorem in the $C^2$ category.

Modern developments in the theory of analysis in metric spaces motivate the study of quasiconformal and conformal mappings beyond Riemannian environments. The sub-Riemannian Heisenberg group $\Heis^n$ was historically the first such space in which quasiconformal mapping theory was considered, and remains an important testing ground for ongoing research. Mostow \cite{mos:rigidity} used quasiconformal mappings of the Heisenberg group in the proof of his eponymous rigidity theorem for rank one symmetric spaces. Kor{\'a}nyi and Reimann \cite{kr:qc}, \cite{kr:foundations} undertook a comprehensive development of Heisenberg quasiconformal mapping theory. In particular, in \cite{kr:qc} the authors prove a Liouville theorem for $C^4$ conformal mappings of the first Heisenberg group $\Heis = \Heis^1$ via the boundary behavior of biholomorphic mappings and the Cauchy--Riemann equations in $\C^2$.

The first proof of Liouville's theorem for $1$-quasiconformal maps of the Heisenberg group was by Capogna \cite{cap:regularity1}. Capogna's proof, similar to those of Gehring and others in the Euclidean setting, relied on nonlinear potential theory, specifically, regularity estimates for $Q$-harmonic functions (here $Q$ is the homogeneous dimension of the Heisenberg group). More recent developments include the work of Capogna--Cowling \cite{cc:conformality} (smoothness of $1$-quasiconformal maps in all Carnot groups), Cowling--Ottazzi \cite{co:conformal} (classification of conformal maps in all Carnot groups), and Capogna--Le Donne--Ottazzi \cite{CLDO} (smoothness of $1$-quasiconformal maps of certain sub-Riemannian manifolds).

In this paper we return to the setting of the Heisenberg group $\Heis^n$. Our aim is to give a new proof of the following theorem.

\begin{theorem}\label{th:main}
Every $C^2$ smooth conformal mapping between domains of the Heisenberg group $\Heis^n$ is $C^\infty$ smooth.
\end{theorem}

Our proof differs from previous proofs in the literature by making no use of nonlinear potential theory, nonlinear PDE, or the boundary behavior of biholomorphic mappings. The only tools which we use are hypoellipticity of H\"ormander operators and quasiconformal flows. Our method is inspired by, but differs in important respects from, the work of Liu and Sarvas.

For the benefit of the reader we provide a brief sketch of the proof of Theorem \ref{th:main} in the setting of the lowest dimensional Heisenberg group $\Heis$. Let $F:U \subset \Heis \to U' \subset \Heis$ be a $C^2$ conformal mapping between domains in the Heisenberg group, write $F=(f_1,f_2,f_3)$ in coordinates, and let $X$, $Y$ and $T$ denote the canonical left invariant vector fields spanning the Lie algebra of $\Heis$. Let $Z=\tfrac12(X-\bi Y)$ and $\overline{Z}=\tfrac12(X+\bi Y)$ be the complexified left invariant horizontal vector fields derived from $X$ and $Y$. We differentiate the conformal flow
$$
s \mapsto F^{-1}(\exp(sT)*F)
$$
at $s=0$ and use hypoellipticity of the H\"ormander operators $-\tfrac14(X^2+Y^2)\pm\sqrt{3}T$ to deduce smoothness of the horizontal Jacobian $J_0F$. In fact, we show that $u=(J_0 F)^{-1}$ is a distributional solution of the equation $ZZu=0$ and appeal to the identity
$$
\tfrac12(\overline{Z}\overline{Z}ZZ+ZZ\overline{Z}\overline{Z}) = \biggl( -\tfrac14(X^2+Y^2) + \sqrt{3}T \biggr) \biggl( -\tfrac14(X^2+Y^2) - \sqrt{3}T \biggr)
$$
to conclude that $u$ is smooth. We then repeat the argument for the conformal flows
$$
s \mapsto F^{-1}(\exp(sX)*F)
$$
and
$$
s \mapsto F^{-1}(\exp(sY)*F)
$$
to deduce that $ZZ(f_j/J_0 F)=0$ for $j=1,2$. Hence $f_1/J_0 F$ and $f_2/J_0 F$ are smooth. It follows that $f_1$ and $f_2$ are smooth, after which smoothness of $f_3$ follows from the contact property of Heisenberg conformal mappings.

The proof in higher dimensional Heisenberg groups follows a similar line of reasoning but uses all possible complexified horizontal second derivatives $Z_kZ_\ell$ and the symplectic structure of the horizontal tangent spaces.

\section{Background and definitions}

\subsection{The Heisenberg group}

We model the $n$th Heisenberg group $\Heis^n$ as $\R^{2n+1}$ equipped with the nonabelian group law
$$
(x,y,t)*(x',y',t') = (x+x',y+y',t+t'-2x\cdot y'+2x'\cdot y),
$$
where $(x,y,t),(x',y',t') \in \R^n \times \R^n \times \R$. Sometimes it is convenient to introduce the complex coordinate $z=x+\bi y \in \C^n$. Then $\Heis^n$ is modeled as $\C^n\times\R$ with group law
$$
(z,t)*(z',t') = (z+z',t+t'+2\omega(z,z')),
$$
where $\omega(z,z') = \Imag(\sum_{j=1}^n z_j\overline{z_j'})$ is the standard symplectic form in $\C^n$ and $z=(z_1,\ldots,z_n)$.

Let $x=(x_1,\ldots,x_n)$ and $y=(y_1,\ldots,y_n)$. The left invariant vector fields
$$
X_j = \deriv{x_j} + 2y_j\deriv{t} \qquad \mbox{and} \qquad Y_j = \deriv{y_j} - 2x_j \deriv{t}
$$
span a $2n$-dimensional subspace $H_p\Heis^n$ in the full tangent space $T_p\Heis^n$ at a point $p=(x,y,t)$. The subbundle $H\Heis^n$ is known as the {\it horizontal bundle}; it defines the accessible directions at $p$. Since the distribution $H\Heis^n$ is nonintegrable --- note that
\begin{equation}\label{eq:commutation-relation}
[X_j,Y_j]=-4T
\end{equation}
for any $j=1,\ldots,n$ where $T=\deriv{t}$ --- the Chow--Raskevsky theorem implies that $\Heis^n$ is horizontally connected. The {\it Carnot--Carath\'eodory metric} $d_{cc}$ is defined as follows: $d_{cc}(p,q)$ is the infimum of the lengths of absolutely continuous horizontal curves $\gamma$ joining $p$ to $q$. Horizontality of $\gamma$ means that $\gamma'(s) \in H_{\gamma(s)}\Heis^n$ for a.e.\ $s$. Length of a horizontal curve is measured with respect to the smoothly varying family of inner products defined in the horizontal subbundle making $X_1,Y_1,\ldots,X_n,Y_n$ into an orthonormal frame. It is well known that $(\Heis^n,d_{cc})$ is a geodesic and proper metric space. The metric $d_{cc}$ is topologically equivalent to the underlying Euclidean metric on $\R^{2n+1}$, but $d_{cc}$ is not bi-Lipschitz equivalent to any Riemannian metric on $\R^{2n+1}$.
%We denote by $B(p,r)$ the ball in the metric $d_{cc}$ with center $p$ and radius $r>0$.

The Lie algebra $\fheis_n$ of $\Heis^n$ can be identified with the tangent space at the origin. Abusing notation, we write $X_j$, $Y_j=X_{n+j}$ and $T$ for the values of the corresponding vector fields at the origin, and note that these elements form a basis for $\fheis_n$. We will denote by $\exp$ the exponential mapping from $\fheis_n$ to $\Heis^n$. Since $\Heis^n$ is connected and simply connected, $\exp$ is a global diffeomorphism.

The first-order differential operators $X_j$ and $Y_j$ are self-adjoint. The {\it Laplacian} (sometimes known as the {\it Kohn Laplacian}) on $\Heis^n$ is the operator
$$
\triangle_0 = \sum_{j=1}^n X_j^2+Y_j^2.
$$
For any $c \in \R$, the operator
\begin{equation}\label{Lc}
\cL_c := -\tfrac14\triangle_0 + c T
\end{equation}
is of H\"ormander type and hence is hypoelliptic. That is, if $\cL_c u = f$ and $f \in C^\infty$, then $u \in C^\infty$. See, e.g., \cite{hor:hypoelliptic}.

The {\it horizontal gradient} of a function $u:\Heis^n \to \R$ is
$$
\nabla_0 u = (X_1u,Y_1u,\ldots,X_nu,Y_nu)
$$
and the {\it horizontal divergence} of a horizontal vector field $\vec{V} = a_1X_1+b_1Y_1+\cdots+a_nX_n+b_nY_n$ is
$$
\diver_0 \vec{V} = X_1(a_1) + Y_1(b_1) + \cdots + X_n(a_n) + Y_n(b_n).
$$
Note that $\triangle_0 u = \diver_0(\nabla_0 u)$.

We make extensive use of the complexified first-order differential operators
$$
Z_j = \frac12 \left( X_j - \bi Y_j \right) \qquad \mbox{and} \qquad \overline{Z_j} = \frac12 \left( X_j + \bi Y_j \right), \quad j=1,\ldots,n.
$$
Note that
$$
\tfrac12\triangle_0 = \sum_{j=1}^n Z_j\overline{Z_j}+\overline{Z_j}Z_j.
$$
Moreover,
\begin{equation}\label{4ZjZk}
4Z_jZ_k = (X_jX_k-Y_jY_k) + \bi(X_jY_k+Y_jX_k), \quad j,k=1,\ldots,n.
\end{equation}
For notational convenience, we sometimes write $x_{n+1}=y_1, \ldots,x_{2n} = y_n$ and similarly $X_{n+1} = Y_1,\ldots X_{2n} = Y_n$. However, we continue to denote the final coordinate by $t$ and we write $T = \partial_t$.

\

Let $U$ be a domain in $\Heis^n$. Write $\cD '(U)$ for the real valued distributions on~$U$.

\begin{lemma}\label{lem:ZZ}
Suppose that $\lambda \in \cD '(U)$, and that $Z_jZ_k\lambda = 0$ for all $j,k=1,\ldots,n$. Then $\lambda$ may be identified with a $C^{\infty}(U)$ function.
\end{lemma}

\begin{proof}
Let $\cL_c$ be the operator defined in \eqref{Lc}. An easy computation yields the identity
\[
	\frac{1}{2} \sum_{j,k=1}^n (\bar{Z_j}\bar{Z_k}Z_jZ_k+Z_jZ_k\bar{Z_j}\bar{Z_k}) = \cL_{-\sqrt{n(n+2)}}\cL_{\sqrt{n(n+2)}},
\]
see, e.g., \cite[p.\ 76]{kr:foundations}. Since $\lambda \in \cD '(U)$ is real valued, the hypothesis $Z_jZ_k\lambda = 0$ also implies that $\bar{Z_j}\bar{Z_k}\lambda = 0$ for all $j,k$. Hence
\[
	\cL_{-\sqrt{n(n+2)}}\cL_{\sqrt{n(n+2)}}\lambda = 0.
\]
Since $\cL_{-\sqrt{n(n+2)}}\cL_{\sqrt{n(n+2)}}$ is a product of hypoelliptic operators, it is also hypoelliptic. Thus $\lambda \in C^{\infty}(U)$ as asserted.
\end{proof}

\begin{remark}
The second order differential operators $Z_jZ_k$ also arise in Kor\'anyi and Reimann's theory of Heisenberg quasiconformal flows \cite[Section 5]{kr:foundations}. To wit, if
$$
V := \varphi T  + \tfrac14 \sum_{j=1}^n (X_j\varphi \, Y_j - Y_j\varphi \, X_j)
$$
for some compactly supported $\varphi \in C^\infty(\Heis^n)$, then the flow maps $f_s:\Heis^n \to \Heis^n$, $s\ge 0$, solving the ODE $\partial_s f_s(p) = V(f_s(p))$, $f_0(p)=p$, are quasiconformal. Specifically, $f_s$ is $K(s)$-quasiconformal, where $K(s)$ satisfies
$$
\tfrac12(K(s)+K(s)^{-1}) = 1+n(\exp(s\sqrt2||M||_{HS})-1).
$$
Here $M = (||Z_jZ_k\varphi||_\infty)_{j,k}$ is the $n\times n$ matrix of $L^\infty$ norms of the functions $Z_jZ_k\varphi$, and $||A||_{HS}$ denotes the Hilbert--Schmidt norm of a matrix $A$.
\end{remark}

\subsection{Conformal mappings of the Heisenberg group}

A reference for the material in this section is \cite[Section 2.3]{kr:foundations}. We consider mappings $F:U \to \Heis^n$ where $U$ is a domain in $\Heis^n$. All mappings will be assumed to be diffeomorphisms which are at least $C^1$ smooth. We write $F=(f_1,\ldots,f_{2n+1})$ in coordinates and denote the standard contact form in $\Heis^n$ by
$$
\alpha = dt+2\sum_{k=1}^n (x_k\,dx_{n+k}-x_{n+k}\,dx_k).
$$
A diffeomorphism $F$ is {\it contact} if it preserves the contact structure. In other words,
\begin{equation}\label{eq:contact-condition}
F^*\alpha = \lambda_F \alpha
\end{equation}
for some nonzero real-valued function $\lambda_F$. We must have either $\lambda_F>0$ everywhere in $U$ or $\lambda_F<0$ everywhere in $U$; we assume that the former condition holds.

Contact maps preserve the horizontal distribution. Denoting by $\bD F(p)$ the differential of $F$ at the point $p$, we have $\bD F(p)(H_p\Heis^n)=H_{F(p)}\Heis^n$ for all $p \in U$. Moreover, the restriction of $\bD F(p)$ to the horizontal tangent space, denoted $\bD_0F(p)$, is a multiple of a symplectic transformation. Indeed, $F^*(d\alpha)|_{H\Heis^n} = \lambda_F d\alpha|_{H\Heis^n}$ and $d\alpha$ defines the standard symplectic structure in $\C^n$. Since $\alpha \wedge (d\alpha)^n$ is a volume form, the full Jacobian $JF = \det \bD F$ satisfies
\begin{equation}\label{eq:Jf-lambda}
JF = \lambda_F^{n+1},
\end{equation}
while the horizontal Jacobian $J_0F = \det \bD_0F$ satisfies
\begin{equation}\label{eq:JHf-lambda}
J_0 F = \lambda_F^n.
\end{equation}
Since $H_p\Heis^n = \Ker \alpha(p)$, \eqref{eq:contact-condition} implies that $\alpha(F_*X_k)=0$ for $k=1,\ldots,2n$, or more explicitly,
\begin{equation}\label{eq:contact-Xk}
X_k f_{2n+1} + 2 \sum_{j=1}^n (f_j X_k f_{n+j} - f_{n+j} X_k f_j) = 0
\end{equation}
for each $k=1,\ldots,2n$. Furthermore,
\begin{equation}\label{eq:contact-T}
Tf_{2n+1} +2 \sum_{j=1}^n (f_j Tf_{n+j} - f_{n+j} Tf_j) = \lambda_F.
\end{equation}

We now come to the definition of the main objects of study in this paper.

\begin{definition}\label{def:conformal}
A $C^1$ diffeomorphism $F:U \to U'$ between domains in $\Heis^n$ is {\it conformal} if $F$ is contact, $J_0F>0$, and the equation
\begin{equation}\label{eq:conformal-PDE}
(\bD_0 F)^T(p) \bD_0 F(p) = J_0 F(p)^{1/n} \, \bI_{2n}
\end{equation}
holds for all $p \in U$.
\end{definition}

In view of \eqref{eq:JHf-lambda}, \eqref{eq:conformal-PDE} can alternatively be written in the form
\begin{equation}\label{eq:conformal-PDE-2}
(\bD_0 F)^T(p) \bD_0 F(p) = \lambda_F(p) \, \bI_{2n}.
\end{equation}
It is known that conformal maps satisfy a Cauchy--Riemann type equation. Defining $f_k^\C = f_k + \bi f_{n+k}$ we have
\begin{equation}\label{eq:conformal-PDE-3}
\overline{Z}_\ell f_k^\C = 0
\end{equation}
for all $1\le k,\ell \le n$. See \cite[Theorem C]{kr:foundations}.

The only properties of conformal mappings which we will use in the proof of Theorem \ref{th:main} are \eqref{eq:conformal-PDE-2} and \eqref{eq:conformal-PDE-3}, together with the facts that inverses and compositions of conformal mappings are conformal. The latter facts are easy to see from the above definition.

Now assume that $F$ is $C^2$. For fixed $\nu=1,\ldots,n$ consider equation \eqref{eq:contact-Xk} for the two indices $\nu$ and $n+\nu$. Differentiating the first using $X_{n+\nu}$ and the second using $X_\nu$ and subtracting yields
\begin{equation*}\begin{split}
& X_{n+\nu} \left( X_\nu f_{2n+1} + 2 \sum_{j=1}^n (f_j X_\nu f_{n+j} - f_{n+j} X_\nu f_j ) \right) \\
& \qquad - X_\nu \left( X_{n+\nu} f_{2n+1} + 2 \sum_{j=1}^n (f_j X_{n+\nu} f_{n+j} - f_{n+j} X_{n+\nu} f_j ) \right) = 0.
\end{split}\end{equation*}
Simplifying the result using \eqref{eq:contact-T} leads to the identity
\begin{equation}\label{lambda-Xjk-identity}
\lambda_F = \sum_{j=1}^n ( X_{n+\nu} f_{n+j} \, X_\nu f_j - X_{n+j} f_\nu \, X_j f_{n+\nu} ),
\end{equation}
valid for each $\nu=1,\ldots,n$. Equation \eqref{lambda-Xjk-identity} can also be derived from the fact that $\bD_0 F$ is a multiple of a symplectic matrix.

\section{Proof of Theorem \ref{th:main}}

Let $F=(f_1,\ldots,f_{2n+1}):U\to U'$ be a $C^2$ conformal mapping between domains of the Heisenberg group $\Heis^n$. Our goal is to prove that $f_j\in C^{\infty}(U)$ for each $j$. Let $G:=F^{-1}$ and write $G=(g_1,\ldots,g_{2n+1})$.

Let $W$ be a left invariant vector field. Fix a domain $\Omega \Subset U$ and choose $s_0>0$ so that the conformal flow
$$
H(p,s) = H_W(p,s) := G(\exp(sW_0)*F(p))
$$
is well defined for all $s$ with $|s|<s_0$ and $p \in \Omega$. Write $H=(h_1,\ldots,h_{2n+1})$.

Denote by $\pi:\Heis^n \to \C^n$ the projection map $\pi(z,t)=z$ and write $\tilde{F} = \pi \circ F = (f_1,\ldots,f_{2n})$. By an application of the chain rule we find that
$$
X_\ell h_k(p,s) = \nabla_0 g_k(\exp(sW_0)*F(p)) \cdot X_\ell \tilde{F}(p)
$$
for each $k,\ell=1,\ldots,2n$ and all $p \in \Omega$. Consequently,
\begin{equation}\label{X-ell-h-k}
X_\ell h_k(p,0) = \nabla_0 g_k(F(p)) \cdot X_\ell \tilde{F}(p) = X_\ell(g_k\circ F)(p) = \delta_{k\ell},
\end{equation}
where $\delta_{k\ell}$ denotes the Kronecker delta.

We now define the matrix-valued flow
$$
\bM(p,s) = \bM_W(p,s) := \lambda_H(p,s)^{-1} (D_0H)^T(p,s) (D_0H)(p,s).
$$
Since $H(\cdot,s)$ is conformal for each $s$, $\bM(p,s) = \bI_{2n}$ for all $p \in \Omega$ and $|s|<s_0$; see \eqref{eq:conformal-PDE-2}. Hence, if $m_{k,\ell}(p,s)$ denotes the $(k,\ell)$ entry of the matrix $\bM(p,s)$, we have
$$
m_{k,\ell}'(p,0) = 0.
$$
Here and henceforth we use primes to denote differentiation with respect to the time parameter $s$. On the other hand,
$$
m_{k,\ell}'(p,s) = \left. \frac{\lambda_H \sum_m \left( (X_kh_m)(X_\ell h_m)' + (X_kh_m)'(X_\ell h_m) \right) - \lambda_H' \sum_m (X_kh_m X_\ell h_m)}{\lambda_H^2} \right|_{(p,s)}.
$$
Since $\lambda_H(p,0) = 1$, we may use \eqref{X-ell-h-k} to conclude that
$$
m_{k,\ell}'(p,0) = (X_\ell h_k)'(p,0) + (X_k h_\ell)'(p,0) - \delta_{k\ell} \lambda_H'(p,0).
$$
We compute $\lambda_H'(p,0)$ using \eqref{lambda-Xjk-identity}. We only need this value in the case $1\le k=\ell \le n$. We choose $\nu = k = \ell$ in \eqref{lambda-Xjk-identity}
%the formula
%\begin{equation*}\begin{split}
%\lambda_H'(p,s) &= \sum_{j=1}^n (X_{n+\nu}h_{n+j})'(X_\nu h_j) + (X_{n+\nu}h_{n+j})(X_\nu h_j)'  \\
%& \qquad \left. - (X_{n+j} h_\nu)'(X_j h_{n+\nu}) - (X_{n+j} h_\nu)(X_j h_{n+\nu})' \right|_{(p,s)}
%\end{split}\end{equation*}
and obtain
$$
\lambda_H'(p,0) = (X_{n+k} h_{n+k})'(p,0) + (X_k h_k)'(p,0).
$$
Hence
$$
m_{k,\ell}'(p,0) = \begin{cases} (X_\ell h_k)'(p,0) + (X_k h_\ell)'(p,0), & k \ne \ell, \\
(X_k h_k)'(p,0) - (X_{n+k} h_{n+k})'(p,0), & 1\le k=\ell \le n.
\end{cases}
$$

In the following lemma we identify the value of $(X_k h_\ell)'(p,0)$ for each $k$ and $\ell$. For a left invariant vector field $W$, we denote by $\tilde{W}$ the {\it right invariant mirror} of $W$, i.e., the unique right invariant vector field whose value at the origin agrees with that of $W$. For instance, if $W = X_j = \partial_{x_j} + 2y_j \partial_t$ then $\tilde{W} = \partial_{x_j} - 2y_j \partial_t$. Observe that
$$
\tilde{X_j} = X_j - 4y_j T, \quad \tilde{Y_j} = Y_j + 4x_j T, \quad \mbox{and} \quad \tilde{T} = T,
$$
and that any of the vector fields $\tilde{X_j}$, $\tilde{Y_j}$ and $\tilde{T}$ commute with all of the left invariant horizontal vector fields $X_k,Y_k$. To see why the latter claim is true, it suffices to verify that $\tilde{X_j}$ commutes with $Y_j$ and that $\tilde{Y_j}$ commutes with $X_j$. In fact,
$$
[\tilde{X_j},Y_j] = X_jY_j - 4y_j TY_j - Y_jX_j + Y_j(4y_j T) = 0
$$
and a similar computation shows that $[\tilde{Y_j},X_j]=0$.

\begin{lemma}\label{lem:RIVF}
For any $k,\ell=1,\ldots,2n$ and with $H(p,s) = H_W(p,s) = (h_1,\ldots,h_{2n+1})(p,s)$, we have
$$
(X_k h_\ell)'(p,0) = X_k(\tilde{W} g_\ell \circ F)(p).
$$
\end{lemma}

\begin{proof}[Proof of Lemma \ref{lem:RIVF}]
First, we show that for a real-valued function $u$ on $U$ and a point $q \in \Heis^n$, we have
$$
\frac{d}{ds}(u(\exp(sW_0)*q)) = (\tilde{W}_q u)(\exp(sW_0)*q).
$$
We use the identity $\exp(W_0)*q = q*\exp(\tilde{W}_q)$ to compute
\begin{equation*}\begin{split}
\frac{d}{ds}(u(\exp(sW_0)*q)) &= \lim_{\delta \to 0} \frac{u(\exp(\delta W_0) * \exp(s W_0) * q) - u(\exp(s W_0) * q)}{\delta} \\
&= \lim_{\delta \to 0} \frac{u(\exp(s W_0) * q * \exp(\delta \tilde{W}_q)) - u(\exp(s W_0) * q)}{\delta} \\
&= (\tilde{W}_q u)(\exp(sW_0)*q).
\end{split}\end{equation*}
We apply the preceding identity with $u = g_\ell$ and $q = F(p)$ to conclude that
$$
h_\ell'(p,0) = (\tilde{W}g_\ell \circ F)(p).
$$
Finally, since $\tilde{W}$ commutes with $X_k$ we have
$$
(X_k h_\ell)'(p,0) = X_k(\tilde{W}g_\ell \circ F)(p).
$$
The proof is complete.
\end{proof}

We now return to the proof of the theorem. The previous discussion has implied that
$$
m_{k,\ell}'(p,0) = \begin{cases}
X_\ell(\tilde{W} g_k \circ F)(p) + X_k(\tilde{W} g_\ell \circ F)(p), & 1\le k \ne \ell \le 2n, \\
X_k(\tilde{W} g_k \circ F)(p) - X_{n+k}(\tilde{W} g_{n+k} \circ F)(p), & 1\le k=\ell \le n.
\end{cases}
$$
We now suppose that there exists a real-valued function $\psi$ on $U$ such that
\begin{equation}\label{gradient-of-psi}
\tilde{W}g_\ell \circ F = X_{n+\ell}\psi \quad \mbox{and} \quad \tilde{W}g_{n+\ell} \circ F = -X_\ell\psi
\end{equation}
for all $\ell=1,\ldots,n$. Then for $k,\ell=1,\ldots,n$ we have
$$
m_{k,\ell}'(p,0) = (X_\ell X_{n+k} + X_{n+\ell} X_k) \psi(p)
$$
and
$$
m_{k,n+\ell}'(p,0) = (X_{n+k} X_{n+\ell} - X_k X_\ell) \psi(p).
$$
Recalling \eqref{4ZjZk} we conclude that
$$
Z_k Z_\ell \psi(p) = 0 \qquad \mbox{for all $p \in \Omega$ and all $1\le k,\ell \le n$.}
$$
Exhausting $U$ with a sequence of compactly contained subdomains $\Omega_\nu$, we conclude that
$$
Z_k Z_\ell \psi = 0 \qquad \mbox{in $U$ for all $1 \le k,\ell \le n$.}
$$
By Lemma \ref{lem:ZZ}, $\psi$ is a $C^\infty$ function.

In order to take advantage of the preceding discussion, we must find an appropriate potential function $\psi$ corresponding to each of the right invariant vector fields $\tilde{X_j}$, $\tilde{Y_j}$ and $\tilde{T} = T$.

First, we consider $\tilde{W} = \tilde{T} = T$. We claim that
$$
\psi = -\frac14 \lambda_F^{-1}
$$
verifies \eqref{gradient-of-psi} for this choice of $\tilde{W}$. Since $\lambda_F^{-1} = \lambda_G \circ F$ it suffices to prove that
\begin{equation}\label{to-prove-1}
Tg_{n+\ell} \circ F = \frac14 X_\ell(\lambda_G \circ F)
\end{equation}
and
\begin{equation}\label{to-prove-2}
Tg_{\ell} \circ F = - \frac14 X_{n+\ell}(\lambda_G \circ F).
\end{equation}
We verify \eqref{to-prove-1}. First
\begin{equation*}\begin{split}
X_\ell(\lambda_G \circ F) &= X_\ell \left( \biggl( Tg_{2n+1} + 2 \sum_{j=1}^n (g_j \, Tg_{n+j} - g_{n+j} \, Tg_j) \biggr)\circ F \right) \\
&= X_\ell(Tg_{2n+1}\circ F) + 2 \sum_{j=1}^n \biggl( X_\ell(g_j \circ F)(Tg_{n+j}\circ F) - X_\ell(g_{n+j} \circ F)(Tg_j \circ F) \biggr. \\
& \qquad \biggl. + (g_j \circ F)X_\ell(Tg_{n+j}\circ F) - (g_{n+j} \circ F) X_\ell (Tg_j \circ F) \biggr) \\
&= (\nabla_0 Tg_{2n+1}\circ F) \cdot X_\ell\tilde{F} \\
& \qquad + 2 \sum_{j=1}^n \biggl( Tg_{n+j} \, \nabla_0 g_j - Tg_j \, \nabla_0 g_{n+j} + g_j \, \nabla_0 Tg_{n+j} - g_{n+j} \, \nabla_0 Tg_j \biggr) \circ F \cdot X_\ell \tilde{F}.
\end{split}\end{equation*}
Since $T$ commutes with $\nabla_0$ we may rewrite this in the form
\begin{equation*}\begin{split}
X_\ell(\lambda_G \circ F) &= (T \nabla_0 g_{2n+1}\circ F) \cdot X_\ell\tilde{F} \\
& \qquad + 2 \sum_{j=1}^n \biggl( Tg_{n+j} \, \nabla_0 g_j - Tg_j \, \nabla_0 g_{n+j} + g_j \, \nabla_0 Tg_{n+j} - g_{n+j} \, \nabla_0 Tg_j \biggr) \circ F \cdot X_\ell \tilde{F}.
\end{split}\end{equation*}
Since $G$ is a contact map,
$$
\nabla_0 g_{2n+1} + 2\sum_{j=1}^n (g_j \nabla_0 g_{n+j} - g_{n+j} \nabla_0 g_j) = 0
$$
and so
\begin{equation*}\begin{split}
X_\ell(\lambda_G \circ F) &= - 2 \sum_{j=1}^n T\biggl( (g_j \nabla_0 g_{n+j} - g_{n+j} \nabla_0 g_j) \biggr) \circ F \cdot X_\ell\tilde{F} \\
& \qquad + 2 \sum_{j=1}^n \biggl( Tg_{n+j} \, \nabla_0 g_j - Tg_j \, \nabla_0 g_{n+j} + g_j \, \nabla_0 Tg_{n+j} - g_{n+j} \, \nabla_0 Tg_j \biggr) \circ F \cdot X_\ell \tilde{F}.
\end{split}\end{equation*}
Using again the fact that $T$ commutes with $\nabla_0$ we conclude that
\begin{equation*}\begin{split}
X_\ell(\lambda_G \circ F) &= 4 \sum_{j=1}^n \biggl( Tg_{n+j} \, \nabla_0 g_j - Tg_j \, \nabla_0 g_{n+j} \biggr) \circ F \cdot X_\ell \tilde{F} \\
&= 4 \sum_{j=1}^n \biggl( (Tg_{n+j}) \circ F \, X_\ell (g_j \circ F) - (Tg_j \circ F) \, X_\ell (g_{n+j} \circ F) \biggr) \\
&= 4 Tg_{n+\ell} \circ F.
\end{split}\end{equation*}
by \eqref{X-ell-h-k}. The proof of \eqref{to-prove-2} is similar. As previously discussed, this shows that the function
$\psi = -\frac14 \lambda_F^{-1}$, and hence $\lambda_F$ itself, is a $C^\infty$ function.

We now consider the right invariant vector field $\tilde{X}_\ell$, for which we claim that the potential function $\psi = f_{n+\ell}\,\lambda_F^{-1}$ verifies \eqref{gradient-of-psi}. We use \eqref{eq:conformal-PDE-2} to deduce that
\begin{equation}\label{X-k-g-ell}
X_\ell g_k \circ F = \lambda_F^{-1} \, X_k f_\ell
\end{equation}
and we use \eqref{eq:conformal-PDE-3} to deduce that
\begin{equation}\label{X-k-f-ell}
X_k f_\ell = X_{n+k} f_{n+\ell}
\end{equation}
for all $1\le k,\ell \le n$. Thus
\begin{equation*}\begin{split}
\tilde{X}_\ell g_k \circ F &= X_\ell g_k \circ F - 4 f_{n+\ell} Tg_k \circ F = \lambda_F^{-1} X_k f_\ell + f_{n+\ell} X_{n+k}(\lambda_F^{-1}),
\end{split}\end{equation*}
where the first line follows from the definition of $\tilde{X}_\ell$ and the second line uses \eqref{X-k-g-ell} and the previous formula for $Tg_k \circ F$. Using \eqref{X-k-f-ell} we conclude that
\begin{equation*}\begin{split}
\tilde{X}_\ell g_k \circ F &= \lambda_F^{-1} X_{n+k} f_{n+\ell} + f_{n+\ell} X_{n+k}(\lambda_F^{-1}) = X_{n+k} (f_{n+\ell} \lambda_F^{-1}).
\end{split}\end{equation*}
A similar computation shows that
$$
\tilde{X}_\ell g_{n+k} \circ F = - X_k (f_{n+\ell} \lambda_F^{-1})
$$
as claimed. As in the previous argument we conclude that $f_{n+\ell} \lambda_F^{-1}$, and hence also $f_{n+\ell}$, are $C^\infty$ for each $1\le \ell \le n$.
Repeating the argument for the right invariant vector fields $\tilde{Y}_\ell = \tilde{X}_{n+\ell}$ shows that the components $f_\ell$ are $C^\infty$ for each $1\le \ell \le n$.

Finally, the contact equation
$$
\nabla_0 f_{2n+1} + 2\sum_{j=1}^n (f_j \nabla_0 f_{n+j} - f_{n+j} \nabla_0 f_j) = 0
$$
implies that $\nabla_0 f_{2n+1}$ is a $C^\infty$ vector field, and hence
$\triangle_0 f_{2n+1} = \diver_0 ( \nabla_0 f_{2n+1} ) \in C^\infty$.
Hypoellipticity of the Kohn Laplacian now implies that $f_{2n+1}$ is $C^\infty$. We have shown that all of the components of $F$ are $C^\infty$ smooth. This completes the proof of Theorem \ref{th:main}.

\begin{remark}
It would be interesting to know if the methods introduced here could be extended to relax the $C^2$ regularity assumption to $C^1$ regularity or even to the Sobolev regularity natural for quasiconformal mappings. Such extension is not without its challenges: for one thing, we differentiate in the vertical and right-invariant directions, and a horizontal Sobolev assumption gives no a priori regularity along these paths. The matter is somewhat subtle, in that one should not be tempted to use the nonlinear theory it was our purpose to avoid. It may be possible to recast the argument, first smoothing some or all of the objects, then justifying the correct limits. Mollification in the sub-Riemannian context has the difficulty that a smoothed contact mapping is likely no longer contact. Depite this, such arguments have been made to work before, and the interested reader might like to consult \cite{dair:bounded-distortion-arb} as a useful starting point.
\end{remark}

%\bibliographystyle{plain}
%\bibliography{references}

\begin{thebibliography}{10}

\bibitem{cap:regularity1}
L.~Capogna.
\newblock Regularity of quasilinear equations in the {H}eisenberg group.
\newblock {\em Comm.\ Pure Appl.\ Math.}, 50(9):867--889, 1997.

\bibitem{cc:conformality}
L.~Capogna and M.~G. Cowling.
\newblock Conformality and {$Q$}-harmonicity in {C}arnot groups.
\newblock {\em Duke Math. J.}, 135(3):455--479, 2006.

\bibitem{CLDO}
L.~Capogna, E.~Le~Donne, and A~Ottazzi.
\newblock Conformality and {Q}-harmonicity in sub-{R}iemannian manifolds.
\newblock Preprint 2016. arXiv:1603.05548v1.

\bibitem{co:conformal}
M.~G. Cowling and A.~Ottazzi.
\newblock Conformal maps of {C}arnot groups.
\newblock {\em Ann.\ Acad.\ Sci.\ Fenn.\ Math.}, 40(1):203--213, 2015.

\bibitem{dair:bounded-distortion-arb}
N.~S. Dairbekov.
\newblock Mappings with bounded distortion on {H}eisenberg groups.
\newblock {\em Sibirsk. Mat. Zh.}, 41(3):567--590, ii, 2000.

\bibitem{geh:rings}
F.~W. Gehring.
\newblock Rings and quasiconformal mappings in space.
\newblock {\em Trans.\ Amer.\ Math.\ Soc.}, 103:353--393, 1962.

\bibitem{hor:hypoelliptic}
L.~H\"ormander.
\newblock Hypoelliptic second order differential equations.
\newblock {\em Acta Math.}, 119:147--171, 1967.

\bibitem{im:gft}
T.~Iwaniec and G.~Martin.
\newblock {\em Geometric function theory and non-linear analysis}.
\newblock Oxford Mathematical Monographs. The Clarendon Press Oxford University
  Press, New York, 2001.

\bibitem{kr:qc}
A.~Kor{\'a}nyi and H.~M. Reimann.
\newblock Quasiconformal mappings on the {H}eisenberg group.
\newblock {\em Invent.\ Math.}, 80(2):309--338, 1985.

\bibitem{kr:foundations}
A.~Kor{\'a}nyi and H.~M. Reimann.
\newblock Foundations for the theory of quasiconformal mappings on the
  {H}eisenberg group.
\newblock {\em Adv.\ Math.}, 111(1):1--87, 1995.

\bibitem{liu:liouville}
Z.~Liu.
\newblock Another proof of the {L}iouville theorem.
\newblock {\em Ann.\ Acad.\ Sci.\ Fenn.\ Math.}, 38(1):327--340, 2013.

\bibitem{mos:rigidity}
G.~D. Mostow.
\newblock {\em Strong rigidity of locally symmetric spaces}.
\newblock Princeton University Press, Princeton, N.J., 1973.
\newblock Annals of Mathematics Studies, No. 78.

\bibitem{resh:bounded-distortion}
Yu.~G. Reshetnyak.
\newblock {\em Space mappings with bounded distortion}, volume~73 of {\em
  Translations of Mathematical Monographs}.
\newblock American Mathematical Society, Providence, RI, 1989.
\newblock Translated from the Russian by H. H. McFaden.

\bibitem{resh:stability}
Yu.~G. Reshetnyak.
\newblock {\em Stability theorems in geometry and analysis}, volume 304 of {\em
  Mathematics and its Applications}.
\newblock Kluwer Academic Publishers Group, Dordrecht, 1994.
\newblock Translated from the 1982 Russian original by N. S. Dairbekov and V.
  N. Dyatlov, and revised by the author, Translation edited and with a foreword
  by S. S. Kutateladze.

\bibitem{sar:liouville}
J.~Sarvas.
\newblock Ahlfors' trivial deformations and {L}iouville's theorem in {${\bf
  R}^{n}$}.
\newblock In {\em Complex analysis {J}oensuu 1978 ({P}roc. {C}olloq., {U}niv.
  {J}oensuu, {J}oensuu, 1978)}, volume 747 of {\em Lecture Notes in Math.},
  pages 343--348. Springer, Berlin, 1979.

\end{thebibliography}

\end{document}